\documentclass{amsart}
\usepackage[T1]{fontenc}
\usepackage[utf8]{inputenc}
\usepackage{lmodern}
\usepackage{amsthm}
\usepackage[utf8]{inputenc}
\usepackage{amssymb,amsmath}
\usepackage{latexsym}
\usepackage[usenames,dvipsnames]{color}
\usepackage[all]{xy}
\usepackage{graphicx}
\usepackage{mathrsfs}
\usepackage{stmaryrd}

\newcommand{\vf}{\mathfrak{X}}
\newcommand{\g}{\mathfrak{g}}
\newcommand{\gs}{{\g^*}{}}
\newcommand{\LL}{\mathscr{L}}
\newcommand{\st}{^*}

\newcommand{\fvf}{\varphi}
\newcommand{\pulce}{\tilde\varphi}
\newcommand{\leaf}{C\big/\mathord\sim}
\newcommand{\id}{\mathrm{id}}
\newcommand{\pr}{\mathrm{pr}}

\newcommand{\Cinfty}{\mathscr{C}^{\infty}}
\newcommand{\R}{\mathbb{R}}
\newcommand{\Z}{\mathbb{Z}}
\newcommand{\Sections}{\Gamma}
\newcommand{\algebra}[1]{\mathscr{#1}}

\DeclareMathOperator{\DD}{D}
\DeclareMathOperator{\dd}{d}

\newtheorem{theorem}{Theorem}[section]
\newtheorem{lemma}[theorem]{Lemma}
\newtheorem{corollary}[theorem]{Corollary}
\newtheorem{proposition}[theorem]{Proposition}
\theoremstyle{definition}
\newtheorem{definition}[theorem]{Definition}
\theoremstyle{remark}
\newtheorem{remark}[theorem]{Remark}

\numberwithin{equation}{section}

\title{Reduction of pre-Hamiltonian actions}

\author[A. De Nicola]{Antonio De Nicola}
 \address{CMUC, Department of Mathematics, University of Coimbra, 3001-501 Coimbra, Portugal}
 \email{antondenicola@gmail.com}

\author[C. Esposito]{Chiara Esposito}
\address{Faculty of Mathematics and Computer Science,
Department of Mathematics,
Chair of Mathematics X (Mathematical Physics), Julius Maximilian University of W\"{u}rzburg, Germany }
\email{chiara.esposito@mathematik.uni-wuerzburg.de}

\begin{document}

\begin{abstract}
We prove a reduction theorem for the tangent bundle of a Poisson
manifold $(M, \pi)$ endowed with a pre-Hamiltonian action of a Poisson
Lie group $(G, \pi_G)$.  In the special case of a Hamiltonian action
of a Lie group, we are able to compare our reduction to the classical
Marsden-Ratiu reduction of $M$.  If the manifold $M$ is symplectic and
simply connected, the reduced tangent bundle is integrable and its
integral symplectic groupoid is the Marsden-Weinstein reduction of the
pair groupoid $M \times \bar{M}$.
\end{abstract}

\maketitle

\section{Introduction}
Reduction procedures for manifolds with symmetries are known in many
different settings.  A quite general approach, whose origin traces
back to the ideas of Cartan~\cite{Cartan1922}, was considered in
\cite{Sniatycki1972} and then generalized in
\cite{Benenti1983,Benenti1982}. In this approach, the reduction of a
symplectic manifold $(M,\omega)$ is intended as a submersion $\rho
:N\to M_{red}$ of an immersed submanifold $i:N\hookrightarrow M$ onto
another symplectic manifold $(M_{red}, \omega_{red})$ such that
$i^*\omega=\rho^*\omega_{red}$.  In particular $M_{red}$ might be the
space of leaves of the characteristic distribution of $i^*\omega$.
However, the most famous result is the one provided by Marsden and
Weinstein \cite{Marsden1974} in the special case where the submanifold
$N$ consists of a level set of a momentum map associated to the
Hamiltonian action a of Lie group.  One of the possible
generalizations has been introduced by Lu \cite{Lu1990} and concerns
actions of Poisson Lie groups on symplectic manifolds. Afterwards, the
case of Poisson Lie groups acting on Poisson manifolds has been
studied in \cite{Esposito2013}. It is also worth to mention the case of
Manin pairs (that include Dirac and Poisson manifolds as special cases)
treated in \cite{Burstzyn2009} and the supergeometric setting
studied in \cite{Cattaneo2013}.
In this paper we consider the case of Poisson Lie groups acting on
Poisson manifolds.
Such actions appear naturally in the study of $R$ matrices and they encode
the hidden symmetries of classical integrable systems.
An action of a Poisson Lie group $(G, \pi_G)$ is said to be Poisson Hamiltonian if it is generated by an
equivariant momentum map $J: M \to G^*$. We shall focus on a
further generalization of Poisson Hamiltonian actions. The main idea,
introduced by Ginzburg in \cite{Ginzburg1996}, is to consider only the
infinitesimal version of the equivariant momentum map studied by Lu.
An action induced by such an infinitesimal momentum map is what
we call pre-Hamiltonian.  Any Poisson Hamiltonian action is pre-Hamiltonian.
Conversely, we show that any pre-Hamiltonian action is a Poisson action.
However, pre-Hamiltonian actions are strictly more general of Hamiltonian actions,
as shown in \cite{Ginzburg1996} where concrete examples of  pre-Hamiltonian actions
which are not Poisson Hamiltonian were provided.

We obtain a reduction theorem for the tangent bundle of a Poisson manifold endowed
with a pre-Hamiltonian action of a Poisson Lie group.
First, we show that a pre-Hamiltonian action of a Poisson Lie group $(G, \pi_G)$ on
a Poisson manifold $(M, \pi)$ defines a coisotropic submanifold $C$ of the tangent
bundle $TM$.  Hence, we can build up a reduced space by using the theory
of coisotropic reduction.  In fact, given a coisotropic submanifold
$C$ of $TM$, the associated characteristic distribution allows us to
define a leaf space $\leaf$, which we denote by $(TM)_{red}$.
The coisotropic submanifold $C$ is defined by means of a map
$\tilde\varphi$ from $\g$ to the space of 1-forms on $M$ which
preserves the Lie algebra structures and is a cochain map.
We obtain the following
\begin{theorem}
    Let  $\Phi: G\times M \to M$ be a  pre-Hamiltonian action of a Poisson Lie group $G$ on
  a Poisson manifold $(M, \pi)$.
    Then the reduced tangent bundle $(TM)_{red}$ carries a Poisson structure.
    Moreover $(TM)_{red}\to M/G$ is a Lie algebroid.
\end{theorem}
The proof uses the theory
of coisotropic reduction, the Tulczyjew's isomorphisms
\cite{Tulczyjew1974, Tulczyjew1999} and the theory of tangent
derivations \cite{Pidello1987}.

Given a pre-Hamiltonian action, we compute
explicitly the infinitesimal generator of the tangent lift of the action, as
can be seen in the following
\begin{theorem}
    \label{thm:TangentLiftAction-intro}
    Let $\Phi: G\times M \to M$ be a pre-Hamiltonian action of a
    Poisson Lie group with infinitesimal momentum map $\pulce$.
    \begin{itemize}
	\item[(i)] The infinitesimal generator $\fvf^T$ of the tangent lift of $\Phi$ is given by
	$$
	  \fvf^T(\xi)  = X_{i_{T}{\pulce_\xi}} + \pi^\sharp_{TM} \circ i_{T}{\dd\pulce_\xi}.
	$$
	\item[(ii)] If for each $\xi\in\g$, one has $\dd\pulce_{\xi}=0$, then the lifted
	(infinitesimal) action on $(TM, \pi_{TM})$ is Hamiltonian, with
	fiberwise-linear momentum map defined by $c_\xi = i_{T}{\pulce}_\xi$.
    \end{itemize}
\end{theorem}

Note that in the special case of a symplectic action on a symplectic manifold we always obtain a Hamiltonian
action on the tangent bundle and hence a Marsden-Weinstein reduction of $TM$ can be performed.

In order to relate our reduction to the classical Marsden-Ratiu
reduction, we consider the particular case of a Hamiltonian action and prove the following
\begin{theorem}
  Let $(M , \pi)$ be  a Poisson manifold endowed with a Hamiltonian
  action of a Lie group $G$ and $0\in \g^*$ a regular value of the momentum
  map $J$.
  Then there is a connection dependent isomorphism of vector bundles
  $$
	    (TM)_{red}|_{M_{red}}
	    \cong
	    T (M_{red})+ \mathfrak{\tilde{g}},
  $$
  where $\mathfrak{\tilde{g}}$ is the associated bundle to the
  principal bundle $J^{-1}(0)\to J^{-1}(0)/G$ by the adjoint
  action of $G$ on $\mathfrak{g}$.
\end{theorem}

Furthermore, by using
\cite{Fernandes2009b,Fernandes2009} we provide an interpretation of
the reduced tangent bundle in terms of symplectic groupoids.  In
particular, we consider the case of a symplectic action of a Lie group
$G$ on a symplectic manifold $M$.  On the one hand, we show that in
this case the lifted action on the tangent bundle $TM$ is Hamiltonian
so that we obtain a reduced tangent bundle $(TM)_{red}$ which is a
symplectic manifold.  On the other hand, it follows from
\cite{Mikami1988} that the symplectic action on $(M,\omega)$ can
be lifted to a Hamiltonian action on the corresponding symplectic
groupoid that can be identified with the fundamental groupoid $\Pi (M)
\rightrightarrows M$ of M.  This implies that the symplectic groupoid
can be reduced via Marsden-Weinstein procedure to a new symplectic
groupoid $(\Pi (M))_{red}\rightrightarrows M/G$. We prove that in this
case our reduced tangent bundle $(TM)_{red}$ is the Lie algebroid
corresponding to the reduced symplectic groupoid $(\Pi (M))_{red}$.
More precisely,
\begin{theorem}
  Given a free and proper symplectic action of a Lie group $G$ on
  a symplectic manifold $(M, \omega)$, we have
  $$
      A((\Pi(M))_{red})
      \cong
      (TM)_{red}.
  $$
\end{theorem}
If $M$ is simply connected, this is just the reduction of the pair
groupoid $M \times \bar{M}$.


\subsection*{Acknowledgements}
  The authors are grateful to Stefan Waldmann and to Henrique Burstzyn
  for their inspiring seminars about coisotropic reduction and
  symplectic groupoids.  Thanks to O\v{g}ul Esen, Yannick Voglaire and
  Ivan Yudin for their helpful comments.

\subsection*{Funding}
Research partially supported by CMUC -- UID/MAT/00324/2013, funded by the Portuguese Government through FCT/MCTES and co-funded by the European Regional Development Fund through the Partnership Agreement PT2020, by MICINN (Spain) grants MTM2011-15725-E, MTM2012-34478 (A.D.N.).


\section{Hamiltonian actions and coisotropic reduction}
\label{sec:Hamiltonian}

  In this section we recall some well-known results regarding
  reduction procedures for Hamiltonian actions and for the more
  general case of coisotropic submanifolds which will be used in the
  following sections.

  Let $G$ be a Lie group and $(M , \pi)$ a Poisson manifold. An action
  $\Phi: G \times M \to M$ is said to be \textbf{canonical} if it
  preserves the Poisson structure $\pi$ on $M$.
  Let $\fvf: \g \to \Gamma(TM)$ be the infinitesimal generator of the action.
  In order to perform a reduction we need to introduce the notion of momentum map.
  \begin{definition}
      A \textbf{momentum map} for a canonical action of $G$ on $M$ 
      is a map $J:M \to \g^*$
      such that it generates the action by
      $$
      \fvf(\xi) = \pi^\sharp(\dd J_\xi),
      $$
      where $J_\xi:M\to \R$ is defined by $J_\xi(p)=\langle J(p),\xi\rangle$, for any $p\in M$ and $\xi\in\g$.
    \end{definition}
A momentum map $J:M \to \g^*$ is said to be \textbf{equivariant} if it
is a Poisson map, where $\g^*$ is endowed with the so-called Lie
Poisson structure \cite[Sec. 3]{CannasdaSilva1999}.  Finally, a
canonical action is said to be \textbf{Hamiltonian} if it is generated
by an equivariant momentum map.

  A reduction theorem for symplectic manifolds with respect to
  Hamiltonian group action was proven in \cite{Marsden1974}. It
  extends in a straightforward way to the case of Poisson manifolds
  which we now recall.
  \begin{theorem}[\cite{Marsden1986}]
  \label{thm:MR}
    Let $(M , \pi)$ be  a Poisson manifold endowed with a free and proper Hamiltonian
    action of a Lie group $G$  and assume that $0\in \g\st$ is a regular value
    for the momentum map $J:M \to \g^*$. Then the reduced space
    $$
    M_{red} = J^{-1}(0)/G
    $$
    is a Poisson manifold.
  \end{theorem}
  Now we briefly review a more general procedure, called coisotropic
  reduction.  The main idea, due to Weinstein \cite{Weinstein1988}, is
  that given a Poisson manifold and a coisotropic submanifold, one can
  always build up a reduced space.  Some nice reviews of this theory can
  be found in \cite{Bordemann2005}, \cite{Bordemann2000} and
  \cite{Cattaneo2005}.  The reduction of a Poisson manifold with respect
  to a Hamiltonian group action as well as coisotropic reduction can
  be recovered as special cases of reduction by distributions
  \cite{Marsden1986,Falceto2008,Jotz2009}.

  Let $(M, \pi)$ be a Poisson manifold and $C\subseteq M$ a submanifold.
  We denote by
  \begin{equation}
    \mathscr{I}_C = \{ f\in \Cinfty(M) : f|_C =0 \}
  \end{equation}
  the multiplicative ideal of the Poisson algebra $\Cinfty(M)$. It is known that
  $C$ is coisotropic if and only if $\mathscr{I}_C $ is a Poisson subalgebra.
  From now on, we assume that $C$ is a regular closed submanifold, so we have
  the identification
  $$
    \Cinfty(M)/\mathscr{I}_C \cong \Cinfty(C).
  $$
   Assume that $(M, \pi)$ is a Poisson manifold and $C \subseteq M$ is a coisotropic
   submanifold.   From the properties of coisotropic manifolds, we know that there always exists
   a characteristic distribution on $C$, which is spanned by the Hamiltonian vector
  fields $X_f$ associated to $f \in \mathscr{I}_C $. This distribution is integrable, so we can define
  the leaf space
  \[
  M_{red}\,:=\, \leaf\,.
  \]
  We assume that the corresponding foliation is simple, that is $M_{red}$ is a
  smooth manifold and the projection map
  \begin{equation}
    p: C \to M_{red}
  \end{equation}
  is a surjective submersion. The manifold $M_{red}$ is called the
  \textbf{reduced manifold}.  One can show that $M_{red}$ is a Poisson
  manifold. More precisely, the following results hold (see
  \cite{Sniatycki2013,Sniatycki1983}).
  \begin{proposition}
    \label{prop:AlgebraicReduction}
    Let $(M, \pi)$ be a Poisson manifold and $C \subseteq M$ a coisotropic
    submanifold.
    \begin{itemize}
	\item[(i)] $\mathscr{B}_C := \{ f\in \Cinfty(M) : \{f,
          \mathscr{I}_C \}\subseteq \mathscr{I}_C \}$ is a Poisson
          subalgebra of $\algebra{A}$ containing $\mathscr{I}$.
	\item[(ii)] $\mathscr{I}_C \subseteq \mathscr{B}_C$ is a
          Poisson ideal and $\mathscr{B}_C$ is the largest Poisson
          subalgebra of $\Cinfty(M)$ with this feature
	\item[(iii)] $\Cinfty(M)_{red} := \mathscr{B}_C/\mathscr{I}_C$ is a Poisson algebra.
    \end{itemize}
  \end{proposition}
  The relation between $\Cinfty(M_{red})$ and $\Cinfty(M)_{red}$ is
  given by the following theorem.
  \begin{theorem}
    \label{thm:CoisotropicReduction}
    Let $M$ be a Poisson manifold and $C$ a closed regular coisotropic
    submanifold defining a simple foliation, so that
    $$
	p: C \to   M_{red}
    $$ is a surjective submersion. Then there exists a Poisson
        structure on $M_{red}$ such that $\Cinfty (M)_{red}$ and
        $\Cinfty (M_{red})$ are isomorphic as Poisson algebras.
  \end{theorem}
  This proves that $M_{red}$ is a Poisson manifold.
  Finally, note that the coisotropic reduction admits as a special case
  the reduction with respect a Hamiltonian group action.
  In this case, the coisotropic submanifold is given by
  the preimage of a regular value of a momentum map. More precisely,
  consider a canonical action $\Phi: G \times M \to M$ generated by an
  $ad^*$-equivariant momentum map $J : M \to \g^*$.
  If $\mu \in \g^*$ is a regular value of $J$ and is $ad^*$-invariant, then
  \begin{equation}
	C_\mu = J^{-1}( \mu)\subseteq  M
  \end{equation}
  is either empty or a coisotropic submanifold. Then the leaf space
  $C_\mu \big/\mathord\sim$ coincides with the orbit space $C_\mu/G$
  (see e.g. \cite{Bordemann2000}), so we get the reduced space of
  Theorem \ref{thm:MR}.


\section{Pre-Hamiltonian actions}\label{sec:mm}

In this section we introduce a generalization of Hamiltonian actions
in the setting of Poisson Lie groups acting on Poisson manifolds.  For
this reason we first recall some basic notions.  A \textbf{Poisson Lie
  group} is a pair $(G,\pi_G)$, where $G$ is a Lie group and $\pi_G$
is a multiplicative Poisson structure.  The corresponding
infinitesimal object is given by a \textbf{Lie bialgebra}, i.e. the
Lie algebra $\g$ corresponding to the Lie group $G$, equipped with the
1-cocycle,
\begin{equation}
    \delta = \dd_e\pi_G : \g\rightarrow  \g\wedge \g.
\end{equation}
If $G$ is connected and simply connected there is a one-to-one
correspondence between Poisson Lie groups and Lie bialgebras (known as
Drinfeld's principle \cite{Drinfeld1983}).  When $(\g,\delta)$ is a
Lie bialgebra, the $1$-cocycle $\delta$ gives a Lie algebra structure
on $\gs$, while the Lie bracket of $\g$ gives a $1$-cocycle
$\delta^*$on $\g^*$, so that $(\gs,\delta^*)$ is also a Lie
bialgebra. Thus, we can define the \textbf{dual Poisson Lie group}
$(G^*, \pi_{G^*})$ as the (connected and simply connected) Poisson Lie
group associated to the Lie bialgebra $(\gs,\delta^*)$. From now on we
assume $G$ to be connected and simply connected in order to get the
one-to-one correspondence stated above.
\begin{definition}
  An action of $(G,\pi_G)$ on $(M,\pi)$ is said to be a
  \textbf{Poisson action} if the map $\Phi:G\times M\rightarrow M$ is
  Poisson, that is
  \begin{equation}
      \lbrace f \circ \Phi, g \circ \Phi \rbrace_{G \times M} =
      \lbrace f,g \rbrace_M \circ \Phi, \qquad \forall f,g\in
      \Cinfty(M)
  \end{equation}
  where the Poisson structure on $G\times M$ is given by $\pi_G\oplus\pi$.
\end{definition}
It is evident that the above definition generalizes the notion of canonical
action.

\begin{definition}[\cite{Lu1990,Lu1991}]
    \label{def:MomentumMap}
    A \textbf{momentum map} for the Poisson action $\Phi:G\times
    M\rightarrow M$ is a map $J: M\rightarrow G^*$ such that
    \begin{equation}
    \label{eq:MomentumMap}
	\fvf(\xi) = \pi^{\sharp}(J^*(\theta_{\xi})),
    \end{equation}
    where $\fvf: \g \to \Gamma(TM)$ denotes the infinitesimal
    generator of the action, $\theta_{\xi}$ is the left invariant
    1-form on $G^*$ defined by the element $\xi \in \g = (T_eG^*)^*$
    and $J^*$ is the cotangent lift of $J $.
\end{definition}

\begin{definition}
    Let $J : M \to G^*$ be a momentum map of the action $\Phi$. Then,
    \begin{itemize}
	\item[(i)] $J$ is said to be \textbf{$G$-equivariant} if it is
          a Poisson map, i.e.
	\begin{equation}
	    J_*\pi = \pi_{G^*},
	\end{equation}
	\item[(ii)] $\Phi$ is said to be a \textbf{Poisson Hamiltonian
          action} if it is a Poisson action induced by a
          $G$-equivariant momentum map.
    \end{itemize}
\end{definition}
This definition generalizes Hamiltonian actions in the canonical
setting. Indeed, we notice that, if the Poisson structure on $G$ is
trivial, the dual $G^*$ corresponds to the dual of the Lie algebra
$\g^*$ and the 1-form $J^*(\theta_{\xi})$ is then exact.  Thus, it
recovers the usual definition of momentum map $J : M \to \g^*$ for
Hamiltonian actions in the canonical setting since $\fvf(\xi)$ is a
Hamiltonian vector field.  As pointed out by Ginzburg in
\cite{Ginzburg1996}, in many cases it is enough to consider the
infinitesimal version of the $G$-equivariant momentum map, which is a
map from the Lie bialgebra $\g$ to the space of 1-forms on $M$. Recall
that a Poisson structure $\pi$ on a manifold $M$ defines a Lie bracket
$[\cdot,\cdot]_{\pi} $ on the space of 1-forms.

\begin{definition}
    \label{def:InfinitesimalMomentumMap}
    Let $(M,\pi)$ be a Poisson manifold endowed with an action of a
    Poisson Lie group $(G,\pi_G)$.
\begin{enumerate}
\item A \textbf{PG-map} is a linear map
    $
    \pulce:\g\to \Omega^{1} (M)
    $
    such that
        \begin{itemize}
	\item[(i)] 	$\pulce_{[\xi,\eta]}  = [\pulce_{\xi},\pulce_\eta ]_{\pi}$
	\item[(ii)]     $\dd\pulce_{\xi} = \pulce \wedge \pulce \circ \delta(\xi)$.
    \end{itemize}
\item Moreover, if $\pulce$ generates the action, that is
      \begin{equation}\label{eq:InfinitesimalMomentumMap}
	\fvf(\xi) = \pi^{\sharp}(\pulce_{\xi}),
      \end{equation}
 we say that it is an \textbf{infinitesimal momentum map}.
\end{enumerate}
\end{definition}
The existence and uniqueness of the infinitesimal momentum map were
studied in \cite{Ginzburg1996}. In particular, it was shown that an
action of a compact group with $H^2 (\g) = 0$ admits an infinitesimal
momentum map.

We are interested to study reduction for actions that admit an infinitesimal momentum map or just a PG-map.
\begin{definition}
    An action of a Poisson Lie group  on a Poisson manifold  is said to be \textbf{pre-Hamiltonian} if it is generated by an infinitesimal momentum map. 
\end{definition}
It is important to remark that this notion is weaker than that of Poisson Hamiltonian action, as it does
not reduce to the canonical one when the Poisson structure on $G$ is trivial.
In fact, in this case we only have that $\pulce_\xi$ is a closed form, but in general this form is not exact.
However, as we will show, it turns out that every pre-Hamiltonian action is a Poisson action.
Concrete examples of  pre-Hamiltonian actions which are not Poisson Hamiltonian were provided in \cite{Ginzburg1996}.
The study of the conditions in which the infinitesimal momentum $\pulce$ map determines the momentum
map $J$ can be found in \cite{Esposito2012a}.

\begin{remark}
Recall that a \textbf{Gerstenhaber algebra} (see
\cite{Kosmann-Schwarzbach1995}) is a triple \( (A=\oplus_{i\in\Z}
A^{i}, \wedge, [\;,\;]) \) such that $(A,\wedge)$ is a graded
commutative associative algebra, \( (A=\oplus_{i\in\Z} A^{(i)},
[\;,\;]), \) with $A^{(i)}=A^{i+1}$, is a graded Lie algebra, and for
each $a\in A^{(i)}$ one has that $[a,\;]$ is a derivation of degree
$i$ with respect to $\wedge$.  A \textbf{differential Gerstenhaber
  algebra} \( (A=\oplus_{i\in\Z} A^{i}, \dd, \wedge, [\;,\;]) \) is a
Gerstenhaber algebra equipped with a differential $\dd$, that is a
derivation of degree 1 and square zero of the associative product
$\wedge$.  One speaks of a \textbf{strong differential Gerstenhaber
  algebra} if, moreover, $\dd$ is a derivation of the graded Lie
bracket $[\;,\;]$.  A \textbf{morphism of differential Gerstenhaber
  algebras} is a cochain map that respects the wedge product and the
graded Lie bracket.  It was proven in \cite{Kosmann-Schwarzbach1995}
that there is a one to one correspondence between Lie bialgebroids and
strong differential Gerstenhaber algebras.  Let $(M,\pi)$ be a Poisson
manifold and $(G,\pi_G)$ a Poisson Lie group.  Then by
\cite{Kosmann-Schwarzbach1995} one has that $(\wedge^{\bullet}\g
,\delta, \wedge, [\;,\;])$ and $(\Omega^{\bullet} (M), \dd_{DR},
\wedge, [\;,\;]_\pi )$ are strong differential Gerstenhaber algebras.
It is easy to check that the notion of infinitesimal momentum map can
be rephrased as a morphism of differential Gerstenhaber algebras
    \begin{equation}
	\label{eq:GInfinitesimalMomentumMap}
	\pulce: (\wedge^{\bullet}\g ,\delta, \wedge, [\;,\;])\longrightarrow (\Omega^{\bullet} (M), \dd_{DR}, \wedge, [\;,\;]_\pi ).
    \end{equation}
However, notice that in general, despite being a morphism of
differential Gerstenhaber algebras, an infinitesimal momentum map
$\pulce: \g \to \Omega^1(M)$ does not always correspond to a morphism
of vector bundles $\g \to T^*M$ and hence it does not necessarily
induce a morphism of Lie algebroids from $\g$ to $T^*M$.
    \end{remark}


\subsection{Properties of PG-maps}\label{ssec:propPG}

The notion of a PG-map is crucial in order to prove a reduction
theorem in this context. For this reason in this section we study some
of its properties. In particular, we can prove that any PG-map (and
hence any infinitesimal momentum map) defines a Lie bialgebroid
morphism. Let us first recall the definitions related with Lie
algebroids that we use in this paper.
\begin{definition}
    \label{def: LieAlgOidMorphism}
    Let $E \to M$ and $F \to N$ be two Lie algebroids.
    A \textbf{Lie algebroid morphism} is a bundle map $\Phi : E \to F$ such
    that
    $$
	\Phi^*:(\Gamma(\wedge^\bullet F^*), \dd^F)\to (\Gamma(\wedge^\bullet E^*),\dd^E)
    $$
    is a cochain map.
\end{definition}

\begin{definition}
    Assume that $E \to M$ is a Lie algebroid and that its dual vector bundle $E^* \to M$
    also carries a structure of Lie algebroid. The pair $(E, E^*)$ of Lie algebroids is
    a \textbf{Lie bialgebroid} if these differentials are derivations of the corresponding
    Schouten brackets, i.e. for any $X,Y \in \Sections(E)$
    \begin{equation}
	\dd_*[X,Y] = [\dd_*X, Y] + [X, \dd_*Y].
    \end{equation}
\end{definition}
It is important to mention that given a Lie bialgebroid $(E, E^*)$,
the Lie algebroid structure on $E$ always induces a linear Poisson
structure on $E^*$ and viceversa. The most canonical example is given
by the Lie bialgebroid $(TM, T^*M)$ associated to a Poisson manifold
$M$. In particular, given a Poisson manifold $M$, its tangent bundle
carries a linear Poisson structure as shown in the following lemma.
\begin{lemma}[{\cite[Prop. 10.3.12]{Mackenzie2005}}]
    Let $(M, \pi_M)$ be a Poisson manifold. Then its tangent bundle
    $TM$ has a linear Poisson structure $\pi_{TM}$ defined by
    \begin{equation}\label{eq:TangentLift}
	\pi^\sharp_{TM} \circ \alpha_M = k_M \circ T\pi^\sharp_{M}, 
    \end{equation}
    where $k_M : TTM \to TTM$ is the canonical involution of the double tangent bundle and $\alpha_M : TT^*M \to T^*TM$ is the Tulczyjew isomorphism  \cite{Tulczyjew1974,Tulczyjew1999}.
\end{lemma}

Note that the linear Poisson structure $\pi_{TM}$ on $TM$ coincides with the standard complete lift of $\pi_{M}$.

We can now give the needed definition of a morphism of Lie bialgebroids.
\begin{definition}
    A \textbf{Lie bialgebroid morphism} is a Lie algebroid morphism
    which is a Poisson map.
\end{definition}

In order to prove that any PG-map $\pulce$ (see
Def.~\ref{def:InfinitesimalMomentumMap}), corresponds to a Lie
bialgebroid morphism, we need to introduce a dual notion to that of
PG-map.
\begin{definition}
    \label{def:ComomentumMap}
    Given a map  $\pulce : \g \to \Omega^1(M)$,
    we can associate the map
    $c: TM \to \g^*$ defined by
    \begin{equation}
	\label{eq:ComomentumMap}
	\langle c(X_m), \xi \rangle = \langle X_m , \pulce_\xi(m)\rangle,
    \end{equation}
    for any $X_m\in T_m M$. If $\pulce$ is an infinitesimal momentum map we call $c$ a \textbf{comomentum map}.
\end{definition}

We are now ready to prove the announced result.

\begin{proposition}
  \label{prop:bialg}
  Let $\pulce : \g \to \Omega^1(M)$ be a PG-map.
  The associated map $c: TM \to \g^*$ is a Lie bialgebroid morphism.
\end{proposition}
\begin{proof}
   From the definition it follows immediately that $c$ is a morphism of vector bundles.
    Indeed, being a vector bundle over a point, $\g^*$  has just one fiber, hence $\pulce$
    sends fibers into fibers. Moreover, $c$ is fiberwise linear, due to the linearity of
    $\pulce$.
    Finally, the pull-back  $c^* : \Gamma(\wedge^\bullet \g) \to \Gamma(\wedge^\bullet T^*M)$
    is given by the the natural extension of the map $\pulce$ and hence it is a cochain map.
    Thus, $c$ is a morphism of Lie algebroids.
    It remains to prove that the map $c^*: \Cinfty(\g^*) \to \Cinfty(TM)$ is a Poisson map,
    i.e. $\{f,g\}_{\g^*}\circ c = \{ f\circ c , g\circ c\}_{TM}$.
    First, we consider $f$ and $g$ to be linear maps from $\g^*$ to $\mathbb{R}$,
    so can denote them as
    \begin{equation*}
	f = l_\xi, \quad g = l_\eta
    \end{equation*}
    for $\xi,\eta \in \g$. For any $\xi \in \g$, we now define
    \begin{equation*}
	l_{\xi^\dagger} := l_\xi\circ c.
    \end{equation*}
    So, we aim to prove that
    \begin{equation*}
	\{ l_{\xi^\dagger}, l_{\eta^\dagger}\} = l_{[\xi,\eta]^\dagger }.
    \end{equation*}
    Using the definition of $c$ it is evident that
    \begin{equation*}
	l_{\xi^\dagger}(v_m) = \pulce_\xi(v_m),
    \end{equation*}
    for any $v_m \in T_m M$. Thus
    \begin{equation*}
	l_{\xi^\dagger} = \pulce_\xi.
    \end{equation*}
    Hence,
    \begin{equation*}
	\{ l_{\xi^\dagger}, l_{\eta^\dagger}\} = \{\pulce_\xi, \pulce_\eta \} = \pulce_{[\xi,\eta]}=l_{[\xi,\eta]^\dagger }.
    \end{equation*}
    The extension to smooth functions can be easily done.
    In facts, we can immediately extend the result to polynomials
    and it is known that the space of polynomials is dense in the space of smooth functions.
\end{proof}

Given an action of a Poisson Lie group $(G, \pi_G)$ on a Poisson manifold $(M, \pi)$,
its infinitesimal generator is a map $\g \to \Gamma(TM): \xi \mapsto \fvf(\xi)$.
Dualizing this map one gets a map $j: T^*M \to \g^*$.
In  \cite[Prop. 6.1]{Xu1995}, the author proves that the action is Poisson if and
only if $j$ is a Lie bialgebroid morphism.
As a consequence we obtain the following result.
\begin{proposition}
  \label{prop:infpoisson}
  Let $\pulce : \g \to \Omega^1(M)$ be an infinitesimal momentum map.
  Then the action induced by $\pulce$ is a Poisson action.
\end{proposition}
\begin{proof}
  By dualizing \eqref{eq:InfinitesimalMomentumMap} we obtain that in our case
  the dual of the infinitesimal action is
  \begin{equation*}
    \label{eq:jcpi}
        j
        =
        c  \circ \pi^\sharp,
  \end{equation*}
  where $c$ is the associated comomentum map defined by \eqref{eq:ComomentumMap}.
  From Prop.~\ref{prop:bialg} we have that $c$ is a Lie bialgebroid morphism.
  Moreover, it is well-known that $\pi^\sharp$ is a Lie bialgebroid morphism,
  so the composition $j$ is a Lie bialgebroid morphism as well.
  The claim then follows by  \cite[Prop.~6.1]{Xu1995}.
\end{proof}

It is important to remark that the above proposition
implies that a pre-Hamiltonian action is always a
Poisson action.


\section{Reduction of the tangent bundle}

In this section, using the techniques of coisotropic reduction
recalled in Sec.\ref{sec:Hamiltonian} and the properties of
PG-maps, we prove a reduction theorem for the tangent bundle of a
Poisson manifold $(M,\pi)$ under the action of a Poisson Lie group.
It is known that the tangent bundle of a Poisson manifold inherits a
linear Poisson structure.  We will show that a PG-map automatically
produces a coisotropic submanifold of the tangent bundle. Thus, we
obtain a reduced Poisson manifold.  Furthermore, in the special case
in which there is a pre-Hamiltonian action of a Poisson Lie group $(G
, \pi_G)$ on $(M,\pi)$ we study the properties of the tangent lift of
the action and this allows us to prove that the Poisson reduced space
coincides with the $G$-orbit space as in the canonical
setting. Finally, in the classic case of a Hamiltonian action on a
Poisson manifold, we analyze the relation of the reduced tangent
bundle $(TM)_{red}$ and the reduced manifold $M_{red}$ produced by
Theorem~\ref{thm:MR}.


\subsection{Coisotropic and pre-Hamiltonian reduction}

Consider a Lie bialgebra $(\g, \delta)$, a Poisson manifold $(M,\pi)$ and let  $\pulce: \g \to \Omega^1(M)$ be a PG-map.
These ingredients are sufficient to obtain a coisotropic reduction.
In Sec.~\ref{ssec:propPG}, to a PG-map $\pulce$ we associated a dual map $c: TM \to \g^*$ by
\eqref{eq:ComomentumMap} and we proved that it is a Poisson map.

The results on coisotropic reduction recalled in Sec.~\ref{sec:Hamiltonian}  can be immediately applied to this case.
More explicitly, we can prove the following result.
\begin{theorem} \label{thm:TangentReduction}
Let $(M,\pi)$ be a Poisson manifold endowed with a PG-map $\pulce:
\g \to \Omega^1(M)$. Then $C := c^{-1}(0)\subseteq TM$ is a
coisotropic submanifold, where $0\in \g^*$ is a regular
value~of~$c$. Moreover, if $C$ defines a simple foliation on $TM$,
then the reduced manifold $(T M)_{red} \,=\, \leaf $ has a Poisson
structure.
\end{theorem}

\begin{proof}
  The fact that $C$ is a coisotropic submanifold follows immediately
  by the fact that $\{0\}$ is a symplectic leaf in $\g^*$ and from the
  fact that $c$ is a Poisson map, as proved in
  Proposition~\ref{prop:bialg}.

To complete the proof it is enough to apply the coisotropic reduction
Theorem \ref{thm:CoisotropicReduction} to our $C$.
\end{proof}
Now, we want to prove that the reduction in the case of a
Pre-Hamiltonian action of a Poisson-Lie group gives rise to a special
case of the coisotropic reduction obtained above. In the following we
always assume the action to be free and proper.

Assume that we have an infinitesimal momentum map $\pulce: \g \to
\Omega^1(M)$. The associated action is in general not Hamiltonian
unless $\pulce_\xi$ is exact.  However, we will see that if
$\pulce_\xi$ is closed the lifted action on the tangent bundle is
always Hamiltonian.  In order to prove these results, we need some
tools from the theory of derivations along maps \cite{Pidello1987}.


A \textbf{tangent derivation} (that is, a derivation along $\tau_M^*$,
see \cite{Pidello1987}) of degree $p$ is a linear operator $\DD :
\Omega^k(M) \to \Omega^{k+p}(TM)$ such that
\begin{equation}
    \label{eq:Tangent}
    \DD (\omega_1 \wedge \omega_2) =     \DD \omega_1 \wedge \tau_M^* \omega_2+ (-1)^{kp}  \tau_M^* \omega_1 \wedge \DD\omega_2.
\end{equation}
We define $i_{T} : \Omega^k(M) \to \Omega^{k-1}(TM)$ as the tangent
derivation of degree $-1$ such that it is zero on functions and acts
on any 1-form $\theta : M \to T^*M$ by
\begin{equation}
    \label{eq:iT}
    i_{T} \theta (v) = \langle \theta(\tau_{M}(v)),v \rangle,
\end{equation}
for any $v\in TM$.

\begin{remark}
    Notice that given $\pulce:\g\to \Omega^1(M)$ and $c: TM \to \g^*$ we can express the map
\begin{align*}
&\g \to \Cinfty(TM)\\
& \xi \mapsto c_\xi :=c\circ \xi
\end{align*}
 in terms of the tangent derivation $i_T$ defined above.
    We get
    \begin{equation}\label{eq:ComomentumTangentDer}
      c_\xi = i_T \pulce_\xi
    \end{equation}
\end{remark}

Then, one easily obtains that on any $k$-form $\omega$ on $M$,
$$
    i_{T} \omega (w_1,\ldots, w_{k-1}) = \langle \omega(\tau_{TM}(w)),  T\tau_{M}(w_1),\ldots, T\tau_{M}(w_{k-1})\rangle
$$
for any $w_1,\ldots, w_{k-1}\in TTM$ such that $\tau_{TM}(w_1)=\ldots =\tau_{TM}(w_{k-1})$.
If $\omega_1$ is a $k$-form and $\omega_2$ is an $l$-form, from \eqref{eq:Tangent} one has
\begin{equation}
    \label{eq:iTLeibniz}
    i_{T} (\omega_1 \wedge \omega_2) =     i_{T} \omega_1 \wedge \tau_M^* \omega_2+ (-1)^k  \tau_M^* \omega_1 \wedge i_T\omega_2
\end{equation}
One can also define
\begin{equation}
    \label{eq:CartanDT}
    \dd_{T}{\omega} = i_{T} \dd{\omega} + \dd i_{T}{\omega}
\end{equation}
It is easy to check that $\dd_{T} : \Omega^k(M) \to \Omega^{k}(TM)$ is
a tangent derivation of degree $0$ and
\begin{equation}
    \label{eq:dTLeibniz}
    \dd_{T} (\omega_1 \wedge \omega_2) =     \dd_{T} \omega_1 \wedge \tau_M^* \omega_2+ (-1)^k  \tau_M^* \omega_1 \wedge \dd_T\omega_2.
\end{equation}
Note that if $\omega$ is a $k$-form on a manifold $M$, then the $k$-form $\dd_{T}{\omega}$ on $TM$ is just the standard complete lift of $\omega$.
The theory of complete lifts of tensor field to the tangent bundle was developed in \cite{Yano1973}.

The following result is known but a proof is not easily available, so we provide one below.
\begin{lemma}\label{lem:Tit}
    For any 1-form $\theta:M\to T^*M$ on a manifold $M$, one has
    $T\theta : TM \to TT^*M$ and
    \begin{equation}
	\label{eq: AlphaT}
	\alpha_M \circ T \theta = \dd_{T}{\theta}.
    \end{equation}
\end{lemma}
\begin{proof}
In this proof we will make use of the Einstein summation convention.
Let us take suitable local coordinate charts $(q^{i})$ in $M$ and
$(q^{i},v^{j})$ in $TM$ (with $i,j=1, \dots, n$). Then, the 1-form
$\theta$, seen as a map $\theta:M\to TM$ has the following coordinate
expression
\begin{equation}\label{eq:teta}
\theta(q)=(q^{i}, \theta_{j}(q)).
\end{equation}
Hence for any $v\in TM$ of coordinates $(q^{i},v^{j})$ one has
\begin{equation*}
i_{T}\theta(v)=\theta_{i}(q)v^{i}.
\end{equation*}
Thus
\begin{equation}\label{eq:di_T}
\dd i_{T}\theta(v)=(q^{i}, v^{j},\partial_{q_{i}} \theta_{j}(q)v^{j},\theta_{j}(q)).
\end{equation}

On the other hand
\begin{equation*}
\dd \theta=\frac12 (\partial_{q_{i}} \theta_{j}-\partial_{q_{j}} \theta_{i})dq^{i}\wedge dq^{j}.
\end{equation*}
Hence
\begin{equation}
i_{T} \dd\theta(v)= (\partial_{q_{j}} \theta_{i}-\partial_{q_{i}} \theta_{j})v^{j} dq^{i}.
\end{equation}
Thus $\dd_{T}\theta=i_{T} \dd\theta+ \dd i_{T}\theta$ has the following local coordinate expression
\begin{equation}\label{eq:d_Ttheta}
\dd_{T}\theta(v)=(q^{i}, v^{j},\partial_{q_{j}} \theta_{i}(q)v^{j},\theta_{j}(q)).
\end{equation}

On the other hand, from \eqref{eq:teta} one has
\begin{equation*}
T _{q}\theta(v)=v^{i} \dd q^{i} + \partial_{q_{i}}\theta_{j}v^{i} \dd v^{j}.
\end{equation*}
Hence as a map
\begin{equation*}
T \theta(v)
=(q^{i}, \theta_{j}(q),v^{i},\partial_{q_{i}} \theta_{j}(q)v^{i}).
\end{equation*}
Now recall that (see e.g. \cite{Tulczyjew1999})
\begin{equation*}
\alpha_{M}(q^{i},p_{j},\dot{q}^{h},\dot{p_{k}})=(q^{i},\dot{q}^{h},\dot{p_{k}},p_{j}).
\end{equation*}
Hence
\begin{equation}\label{eq:alfaT}
\alpha_{M}\circ T \theta(v)
=(q^{i} ,v^{j},\partial_{q_{j}} \theta_{i}(q)v^{j},\theta_{j}(q)).
\end{equation}
By comparing \eqref{eq:d_Ttheta}  and \eqref{eq:alfaT}, the claim follows.
\end{proof}
\begin{remark}
Note that a similar result to Lemma~\ref{lem:Tit} is
\cite[Theorem~2.1]{Grabowski1995} that relates $i_T(T\theta)$ and
$i_T(\dd_{T}\theta)$ via $k_M$. Now, $k_M$ is in a sense dual to
$\alpha_M$. So one could try to derive the lemma from the result in
\cite{Grabowski1995}.  However, the duality between $k_M$ and
$\alpha_M$ is highly nontrivial, as it involves two different pairings
(see e.g. \cite{Tulczyjew1999}). Hence we think that it is easier to
prove our claim directly, as we did above.
\end{remark}
Given a pre-Hamiltonian action, the above results allow us to compute
explicitly the infinitesimal generator of the tangent lift of the action, as
can be seen in the following
\begin{theorem}
    \label{thm:TangentLiftAction}
    Let $\Phi: G\times M \to M$ be a pre-Hamiltonian action of a
    Poisson Lie group with infinitesimal momentum map $\pulce$.
    \begin{itemize}
	\item[(i)] The infinitesimal generator $\fvf^T$ of the tangent lift of $\Phi$ is given by
	$$
	  \fvf^T(\xi)  = X_{i_{T}{\pulce_\xi}} + \pi^\sharp_{TM} \circ i_{T}{\dd\pulce_\xi}.
	$$
	\item[(ii)] If for each $\xi\in\g$, one has $\dd\pulce_{\xi}=0$, then the lifted
	(infinitesimal) action on $(TM, \pi_{TM})$ is Hamiltonian, with
	fiberwise-linear momentum map defined by $c_\xi = i_{T}{\pulce}_\xi$.
    \end{itemize}
\end{theorem}
\begin{proof}
    \begin{itemize}
     \item[(i)] Since $\pulce$ generates the action, we have the  relation
    $$
	\fvf(\xi) = \pi_M^\sharp \circ \pulce_{\xi}.
    $$
     Moreover (see \cite{Grabowski1995} or \cite[p.365]{Mackenzie2005}),
    $$
	\fvf^T(\xi) = k_M \circ T(\fvf(\xi)).
    $$
    Now, by substituting the first relation in the second one, we obtain
    \begin{align*}
	\fvf^T(\xi)  =  k_M \circ T \pi_M^\sharp \circ T\pulce_\xi .
    \end{align*}
    By using \eqref{eq:TangentLift} and \eqref{eq: AlphaT},   we get
    \begin{align*}
    	\fvf^T(\xi) &= \pi^\sharp_{TM} \circ \alpha_M \circ T\pulce_\xi\\
		 &= \pi^\sharp_{TM} \circ \dd_{T}{\pulce_\xi}\\
		 &= \pi^\sharp_{TM} \circ (\dd i_{T}{\pulce_\xi}+ i_T \dd \pulce_\xi)\\
		 &= X_{i_{T}{\pulce_\xi}} + \pi^\sharp_{TM} \circ i_T \dd \pulce_\xi.
    \end{align*}
    \item[(ii)] Clearly, if $\dd\pulce_{\xi}=0$, we get $\fvf^T(\xi) = X_{i_{T}{\pulce_\xi}}$.
    \end{itemize}
\end{proof}

It is clear that if $\dd \pulce_\xi = 0$ for any $\xi\in \g$, then we
can reduce the tangent bundle by using the well-known
Theorem~\ref{thm:MR} of reduction of Poisson manifolds, since in
Theorem~\ref{thm:TangentLiftAction} we proved that in this case the
tangent lift of the action is Hamiltonian. In other words, the
reduction procedure obtained above recovers the reduction of Poisson
manifolds in the specific case in which the infinitesimal momentum map
associates a closed form to any element of the Lie bialgebra.  In
particular, this happens in the case of a symplectic action on a
symplectic manifold $(M,\omega_M)$. Then, the tangent bundle is also
symplectic, with the symplectic form given by $d_T\omega_{M}$.
\begin{corollary}
    \label{cor:SymplecticCase}
    Let $\g \to \vf(M)$ be a symplectic action of a Lie algebra $\g$ on a symplectic manifold $(M,\omega_M)$.
    Then, the lifted  action on $(TM, d_T\omega_{M})$ is Hamiltonian,
    with fiberwise-linear momentum map $c:TM\to \g^*$defined by $$c_\xi = i_{T}(i_{\fvf(\xi)}\omega_M).$$
\end{corollary}
\begin{proof}
    Under the above assumptions we have that the action is clearly pre-Hamiltonian with infinitesimal momentum map $\pulce: \g \to \Omega^1(M)$   given by
    $$
	\pulce_{\xi}=i_{\fvf(\xi)}\omega_{M}.
    $$
    Moreover, $\LL_{\fvf(\xi)}\omega_M=0$ implies $\dd i_{\fvf(\xi)}\omega_M = 0$ since
    $\omega_{M}$ is closed. 
\end{proof}
Theorem~\ref{thm:TangentLiftAction} allows us to show that, in the case of a (free and proper) pre-Hamiltonian $G-$action, the reduced Poisson manifold $(TM)_{red}\,=\,\leaf$ of Theorem~\ref{thm:TangentReduction}  is the orbit space of the lifted action of $G$ on $C \subseteq TM$.

\begin{theorem}\label{thm:TangentReductionG}
    Let  $\Phi: G\times M \to M$ be a  pre-Hamiltonian action of a Poisson Lie group
    with infinitesimal momentum  map $\pulce$ and comomentum map $c$.
    Then the reduced tangent bundle is given by $(TM)_{red} = C/G$.
    Moreover $(TM)_{red}\to M/G$ is a Lie algebroid.
\end{theorem}
\begin{proof}
    Let $\{e^i\}_{i=1, \dots, n}$  be a basis of $\g^*$ and $c_i : TM\to \R$ the
    components of $c$. Thus,
    $$
	c =  \sum_i c_i e^i.
    $$
    Since $ C = c^{-1}(0)$ and $\mathscr{I}_C$ is the set of functions vanishing on $C$,
    then by \cite[Lemma 5]{Bordemann2000} any $f \in \mathscr{I}_C$ can be written as
    \begin{equation}\label{eq:Xf}
	f =  \sum_i f^i c_i,
    \end{equation}
    where
    $$
	f^i: TM \to \R.
    $$
    Consider the inclusion $i: C \to TM$ and a Hamiltonian vector field $X_f$ on $TM$
    (recall that they span the characteristic distribution on $C$). From \eqref{eq:Xf} we have
    $$
	i^* X_f =  \sum_i i^*(f^i X_{c_i} + c_i X_{f^i}),
    $$
    by the Leibniz rule. The term $i^*(c_i X_{f^i})$ is zero because $c_i$ vanishes on $C$. So we get
    \begin{equation}\label{prf:Xf}
      i^* X_f = \sum_i i^*(f^i X_{c_i}).
    \end{equation}
    From Theorem~\ref{thm:TangentLiftAction} and \eqref{eq:ComomentumTangentDer}, we have
    \begin{equation}\label{eq:fit}
    \fvf^T(e_i) = X_{i_T \pulce_i} + \pi^\sharp_{TM} \circ i_{T}{\dd\pulce_i} = X_{c_i} + \pi^\sharp_{TM} \circ i_{T}{\dd\pulce_i},
    \end{equation}
    where $\pulce_i:=\pulce({e_i})$. We have
    \[
    \delta (e_i) = \sum_{j<k} \gamma^{jk}_i e_{j} \wedge e_{k},
    \]
    where $\gamma^{jk}_{i}$  are some real constants.
    Now, using the property $\dd\pulce_\xi = \pulce \wedge \pulce \circ \delta (\xi)$ we can write
    \[
    i_{T}{\dd\pulce_i} = \sum_{j<k} \gamma^{jk}_i i_{T}({\pulce_j\wedge\pulce_k}).
    \]
    Hence, from \eqref{eq:ComomentumTangentDer} and \eqref{eq:iTLeibniz}  we get
    $$
    i_{T}{\dd\pulce_i} = \sum_{j<k}\gamma^{jk}_i ( c_j \wedge \tau_M^*\pulce_k - \tau_M^*\pulce_j \wedge c_k ).
    $$
    Thus, since the $c_i$\,s' are functions, we have
    $$
    \pi^\sharp_{TM} ( i_{T}{\dd\pulce_i} )=\sum_{j<k}\gamma^{jk}_i( c_j \pi^\sharp_{TM} ( \tau_M^*\pulce_k ) - c_k \pi^\sharp_{TM} ( \tau_M^*\pulce_j) ).
    $$
    From \eqref{eq:fit}, by using this relation we get
    \begin{equation}\label{prf:Xi}
      i^* \fvf^T(e_i) = i^* (X_{c_i} + \pi^\sharp_{TM} (i_{T}{\dd\pulce_i})) = i^* X_{c_i}
    \end{equation}
    because the functions $c_i$\,s' vanish on $C$. Substituting \eqref{prf:Xi} in \eqref{prf:Xf} we
    have
    $$
	i^* X_f =\sum_i i^*(f^i) \cdot i^*\fvf^T(e_i). 
    $$
    We have proved that the leaves of the characteristic distribution are the $G$-orbits.

    It remains to prove that $(TM)_{red}\to M/G$ is a Lie algebroid. %
    over $M/G$ First note that $C=c^{-1}(0)\to M$ is a Lie
    subalgebroid of the tangent bundle, as $c$ is a morphism of Lie
    (bi)algebroids.  Moreover, the free and proper action of $G$ on
    $M$ lifts to free and proper action on $TM$ by Lie algebroid
    automorphisms. Thus, since the lifted action restricts to an
    action by Lie algebroid automorphisms on $C$, one gets a quotient
    Lie algebroid $C/G\to M/G$ (see
    e.g. \cite{Mackenzie2005,Marrero2012}).
\end{proof}

\begin{remark}
  As an example, let us consider the dressing action $G \times G^* \to
  G^*$, which is Poisson Hamiltonian with momentum map $J =
  \id$. Thus, using the linearization $T G^*\cong G^* \times
  \mathfrak{g}^*$ and the definition of the comomentum map
  (\ref{eq:ComomentumMap}), we get
  $$
    c = \pr_{\mathfrak{g}^*}\colon G^* \times \mathfrak{g}^* \to \mathfrak{g}^*.
  $$
  Hence, in this case the reduction of the tangent bundle gives as a result  the space
  of orbits of the dressing action:
  $$
  (TG^*)_{red} = G^*/G.
  $$
\end{remark}

\subsection{Relation with the Hamiltonian case}
Let us consider the particular case in which the pre-Hamiltonian
action is Hamiltonian, so we have a momentum map $J : M \to \g^*$ and
$\pulce_\xi = \dd J_\xi$ (for instance, this occurs if $\pulce_\xi =
J^*(\theta_\xi)$ and $\pi_G=0$).  As recalled in
Section~\ref{sec:Hamiltonian}, in this case we have the well-known
reduction Theorem \ref{thm:MR} which gives a reduced Poisson manifold
$M_{red} = J^{-1}(0)/G$. The following theorem gives the relation
between $M_{red}$ and $(TM)_{red}$.
\begin{theorem}\label{thm:comparison}
Let $(M , \pi)$ be  a Poisson manifold endowed with a Hamiltonian
    action of a Lie group $G$  and assume that $0\in \g\st$ is a regular value
    for the momentum map $J:M \to \g^*$.
Then,  we have
\(
	(TM)_{red} = (\dd_{T}J)^{-1}(0)/G.
\)
Moreover, there is a connection dependent isomorphism of vector bundles
    $$
	(TM)_{red}|_{M_{red}} \cong T (M_{red})+ \mathfrak{\tilde{g}},
    $$ where $\mathfrak{\tilde{g}}$ is the associated bundle to the
principal bundle $J^{-1}(0)\to J^{-1}(0)/G$ by the adjoint action of
$G$ on $\mathfrak{g}$.
\end{theorem}
\begin{proof}
    Recall that
    $$
	M_{red} = J^{-1}(0)/G,
    $$
    where $J: M \to \g^*$ and
    $$
	(TM)_{red} = c^{-1}(0)/G,
    $$
    where $c:TM\to \g^{*}$ is the comomentum map associated to $\pulce$.
    Moreover, we have
    \[
	\pulce_{\xi} = \dd J_{\xi},
    \]
    for each $\xi\in\g$.
    Hence $c_{\xi} : TM \to \R$ is given by
    \begin{equation}\label{eq:cixi}
	c_{\xi} = i_{T}\pulce_{\xi} = i_{T} \dd J_{\xi} = \dd_{T}J_{\xi}.
    \end{equation}
    We can  extend the action of $\dd_{T}$ to act on $\g^*$-valued functions such as $J$ in an obvious way, giving as a result
    $\dd_{T}J : TM \to \g^*$. Then, from \eqref{eq:cixi} we obtain $c=\dd_{T}J$ and hence
    \begin{equation}\label{eq:tmred}
	(TM)_{red} = (\dd_{T}J)^{-1}(0)/G.
    \end{equation}

    Now, we have
    \begin{equation}\label{eq:hope}
	(\dd_{T}J)^{-1}(0)|_{J^{-1}(0)}=T(J^{-1}(0)).
    \end{equation}
    Indeed, if $v\in T(J^{-1}(0))\subset TM$ then $J(\tau_M(v))=0$ and  $\langle \dd J(\tau_M(v)), v \rangle=0$. Hence
    \begin{equation*}
	(\dd_{T}J)(v)=i_T \dd J(v)=0.
    \end{equation*}
    Conversely, if $v\in (\dd_{T}J)^{-1}(0)$ and $J(\tau_M(v))=0$ then $\tau_M(v)\in J^{-1}(0)$ and
    \begin{equation*}
	0=(\dd_{T}J)(v)=i_T \dd J(v)=\langle \dd J(\tau_M(v)), v \rangle.
    \end{equation*}
    Thus, we get $v\in T(J^{-1}(0))$.

    We use now the following fact. Since the lifted action of $G$ on
    $TM$ is free and proper, by \cite[Lemma 2.4.2]{Marsden2001} one
    has a connection dependent isomorphism
    \[
	(TM)/G\cong T (M/G) + \mathfrak{\tilde{g}},
    \]
where $\mathfrak{\tilde{g}}$ is the associated bundle to the principal bundle $M\to M/G$ by the adjoint action of $G$ on $\mathfrak{g}$.
    If we apply the above result to $J^{-1}(0)$,  we get
        \begin{align*}
(TJ^{-1}(0))/G\cong T (J^{-1}(0)/G) + \mathfrak{\tilde{g}}= T(M_{red}) + \mathfrak{\tilde{g}}
    \end{align*}
    and hence by \eqref{eq:hope} we obtain
    \begin{align*}
    (\dd_{T}J)^{-1}(0)|_{J^{-1}(0)}/G\cong T(M_{red}) + \mathfrak{\tilde{g}},
    \end{align*}
    where $\mathfrak{\tilde{g}}$ is the associated bundle to the
    principal bundle $J^{-1}(0)\to J^{-1}(0)/G$ by the adjoint action
    of $G$ on $\mathfrak{g}$.  On the other hand, from
    \eqref{eq:tmred} we have
        \begin{align*}
	(TM)_{red}|_{M_{red}}=(\dd_{T}J)^{-1}(0)/G|_{J^{-1}(0)/G} =(\dd_{T}J)^{-1}(0)|_{J^{-1}(0)}/G,
    \end{align*}
since the tangent projection is equivariant with respect to the considered group actions.  Hence we conclude that
        \begin{align*}
	(TM)_{red}|_{M_{red}}\cong T(M_{red}) + \mathfrak{\tilde{g}}.
    \end{align*}
\end{proof}

This theorem shows that in case of a Hamiltonian action our reduced
tangent bundle is closely related to the tangent bundle of the
classical Marsden-Ratiu reduced manifold.

Note that a part of Theorem~\ref{thm:comparison} in the special case of symplectic manifolds was already proved in \cite[Theorem~4.1]{Garcia2014}.

\section{Integration of the reduced tangent bundle}

In this section we give an interpretation of the reduced tangent
bundle in terms of symplectic groupoids. In particular, we consider
the case of a symplectic action of a Lie group $G$ on a symplectic
manifold $M$.  On the one hand, we have shown that in this case the
lifted action on the tangent bundle $TM$ is Hamiltonian so that we
obtain a reduced tangent bundle $(TM)_{red}$ which is a symplectic
manifold.  On the other hand, we can apply to our case the results of
\cite{Mikami1988} and
\cite{Fernandes2009b} that hold for a canonical action on an
integrable Poisson manifold $(M, \pi)$. Now, every symplectic manifold
is integrable and the corresponding symplectic groupoid that can be
identified with the fundamental groupoid $\Pi(M)$ of $M$.  Hence, the
action of $G$ on $M$ can be lifted to a Hamiltonian action on the
fundamental groupoid $\Pi(M) \rightrightarrows M$.  This implies that
the symplectic groupoid $\Pi(M)$ can be reduced via Marsden-Weinstein
procedure to a new symplectic groupoid $(\Pi(M))_{red}\rightrightarrows
M/G$. We prove that our reduced tangent bundle $(TM)_{red}$ is the Lie
algebroid corresponding to the reduced symplectic groupoid
$(\Pi(M))_{red}$.

First, let us recall the needed results from \cite{Fernandes2009b}.
\begin{theorem}[\cite{Fernandes2009b}]
\label{thm:OidAction}
    Let $G \times M \to M$ a free and proper canonical action on an integrable Poisson
manifold $(M, \pi)$.
    There exists a unique lifted action
    of $G$ on $\Sigma (M) \rightrightarrows M$ by symplectic groupoid automorphisms.
    This lifted action is free and proper and Hamiltonian. Let $J : \Sigma (M) \to \g^*$ denote its momentum map.
    Then, the reduced symplectic
groupoid, given by
    $$
    (\Sigma(M))_{red} = J^{-1}(0)/G
    $$
    is a symplectic groupoid integrating $M/G$.

\end{theorem}
Note that the symplectic form on $(\Sigma(M))_{red}$ allows us to identify
the Lie algebroid $A((\Sigma(M))_{red})$ with the cotangent Lie algebroid $T^*(M/G)$.

We will also use the following well-known result on the cotangent bundle reduction.
\begin{theorem}[\cite{Abraham1978}]\label{thm:CotangentAction} 
    Given a free and proper action of a Lie group $G$ on $M$,
    the cotangent lift of the action is Hamiltonian with momentum map given by
        \[
    \langle j(\alpha_m),\xi\rangle=\alpha_m (\fvf_\xi (m)),
    \]
for any $\alpha_m\in T_m^*M$, $\xi\in\g$.
    Moreover, we have
    $$
    (T^*M)_{red} \cong T^* (M/G).
    $$
\end{theorem}
The following lemma shows that in the case of a symplectic action on a symplectic manifold  $(M,\omega)$,
due to the  isomorphism $\omega^\flat : TM \to T^*M$ induced by  $\omega$, the reduced tangent bundle produced by Corollary~\ref{cor:SymplecticCase} is isomorphic to the classical reduced cotangent bundle.
\begin{lemma}\label{lem:CotangentTangent}
Given a symplectic action $G\times M \to M$ of a Lie group $G$ on a symplectic manifold $(M,\omega)$ we have
    \[
    (TM)_{red} \cong (T^*M)_{red}.
    \]
\end{lemma}
\begin{proof}
The momentum map $j:T^*M\to \g^*$ of the cotangent lift of the action is characterized by
    \[
    j_{\xi}(\alpha_m)=\alpha_m (\fvf_\xi (m)),
    \]
for any $m\in M$, $\alpha_m\in T_m^*M$ and $\xi\in\g$. Hence, by writing $\alpha_m=\omega^\flat{X_m}$ with $X_m \in T_m M$ we get
    \[
    j_{\xi}(\omega^\flat{X_m})=(i_{X_m}\omega)_m (\fvf_\xi (m))=\omega_m ( X_m, \fvf_\xi(m)).
    \]
On the other hand, by Corollary~\ref{cor:SymplecticCase} the momentum map $c:TM\to \g^*$ of the tangent lift of the action is characterized by
    \[
    c_{\xi}(X_m)=i_T (i_{\fvf(\xi)}\omega) (X_m)=(i_{\fvf(\xi)}\omega)_m (X_m)= -\omega_m ( X_m, \fvf_\xi (m)).
    \]
Thus we conclude that
    \[
    c=-j  \circ \omega^\flat.
    \]
    As a consequence,
    \[
    j^{-1}(0)= -\omega^\flat(c^{-1}(0)).
    \]
Since the symplectic form is $G$-invariant,  we conclude that
    \[
    j^{-1}(0)/G \cong c^{-1}(0)/G.
    \]
\end{proof}

Note that Lemma~\ref{lem:CotangentTangent} in the particular case when $M$ is a canonical tangent bundle and the $G$-action is a cotangent lift was already proved in \cite[Lemma~4.2]{Garcia2014}.

These results allow us to prove that in the symplectic action of a Lie group on a symplectic manifold the reduced tangent manifold $(TM)_{red}$
is the Lie algebroid of the reduced symplectic groupoid $(\Pi(M))_{red}\rightrightarrows M/G$.
\begin{theorem}\label{thm:ReducedGroupoid}
 Given a free and proper symplectic action of a Lie group $G$ on a symplectic manifold $(M, \omega)$, we have
 $$
     A((\Pi(M))_{red}) \cong  (TM)_{red}.
 $$
\end{theorem}

\begin{proof}
By Corollary~\ref{cor:SymplecticCase}, we know that the lifted action to $TM$ is Hamiltonian with momentum map $c$.
Thus we can perform the reduction of $TM$ obtaining a reduced tangent bundle
\[
    (TM)_{red} = c^{-1}(0)/G
\]
endowed with a symplectic structure.

From Theorem~\ref{thm:OidAction}, we know that the symplectic action
can be lifted to a Hamiltonian action on $\Sigma (M)$ that in this case can be identified with $\Pi(M)$
and the reduced symplectic groupoid is
\[
   (\Pi(M))_{red} := J^{-1}(0)/G \rightrightarrows M/G,
\]
where $J: \Pi(M) \to \g^*$ is the momentum map of the lifted action.
In this case the base manifold $M/G$ is symplectic. Moreover, we have
\begin{equation}\label{eq:RedAlgebroid}
     A((\Pi(M))_{red}) = T^*(M/G).
\end{equation}
Now, by Theorem \ref{thm:CotangentAction}, we have
\begin{equation}\label{eq:RedAlgebroid-ctg}
     A((\Pi(M))_{red}) \cong (T^*M)_{red}.
\end{equation}
Hence, by using Lemma \ref{lem:CotangentTangent} we get
\[
    A((\Pi(M))_{red}) \cong (TM)_{red}.
\]
\end{proof}
Note that  if $(M, \omega)$ is simply connected the corresponding  symplectic groupoid is the pair groupoid:
\[
M \times \bar{M}\rightrightarrows  M,
\]
with symplectic structure $\omega \oplus (-\omega)$.
Thus, we obtain the following result.
\begin{corollary}
    Let $M$ be symplectic and simply connected and let
    $G\times M \to M$ be a symplectic action. Then,
    \[
    A((M \times \bar{M})_{red}) \cong (TM)_{red}.
    \]
\end{corollary}

\bibliographystyle{acm}

\end{document}